\documentclass[12pt,a4paper]{article}
\usepackage[utf8]{inputenc}
\usepackage[T1]{fontenc}
\usepackage{amsmath,amssymb,amsthm}
\usepackage{cite}
\usepackage{mathtools}
\usepackage{enumitem}

\makeatletter
\renewcommand*{\@fnsymbol}[1]{\ensuremath{\ifcase#1\or 1\or 2\or 3\else\@ctrerr\fi}}
\makeatother

\numberwithin{equation}{section}

\usepackage[pdftex,bookmarksopen=true,bookmarks=true,unicode,setpagesize]{hyperref}
\hypersetup{colorlinks=true,linkcolor=black,citecolor=black,urlcolor=blue}

\newtheorem{theorem}{Theorem}[section]
\newtheorem{condition}{Condition}
\newtheorem{corollary}[theorem]{Corollary}
\newtheorem{lemma}[theorem]{Lemma}
\newtheorem{proposition}[theorem]{Proposition}
\theoremstyle{definition}
\newtheorem{definition}[theorem]{Definition}
\newtheorem{remark}[theorem]{Remark}
\newtheorem{example}[theorem]{Example}

\newcommand\R{{\mathbb{R}}}
\newcommand{\N}{\mathbb{N}}
\newcommand{\B}{\mathbb{B}}
\newcommand\dt{\dfrac{d}{dt}}
\renewcommand\L{{\mathcal{L}}}

\begin{document}

\title{Row-finite systems of ordinary differential equations in a scale of Banach spaces}
\author{Alexei Daletskii\thanks{Department of Mathematics, The University of York, York YO1 5DD, U.K. ({\tt ad557@york.ac.uk}).} \and Dmitri Finkelshtein\thanks{Department of Mathematics,
Swansea University, Singleton Park, Swansea SA2 8PP, U.K. ({\tt d.l.finkelshtein@swansea.ac.uk}).}}
\maketitle

\begin{abstract}
Motivated by the study of dynamics of interacting spins for infinite particle systems, we consider an infinite family of first order differential equations in a Euclidean space, parameterized by elements $x$ of a fixed countable set. We suppose that the system is row-finite, that is, the right-hand side of the $x$-equation depends on a finite but in general unbounded number $n_x$ of variables. Under certain dissipativity-type conditions on the right-hand side and a bound on the growth of $n_x$, we show the existence of the solutions with infinite life-time, and prove that they live in an increasing scale of Banach spaces. For this, we obtain uniform estimates for solutions to approximating finite systems using a version of Ovsyannikov's method for linear systems in a scale of Banach spaces. As a by-product, we develop an infinite-time generalization of the Ovsyannikov method.


\end{abstract}

\section{Introduction}
In recent decades, there has been an increasing interest to the study of countable systems of particles randomly distributed in the Euclidean space $\mathbb{R}^{d}$, which appear, in particular, in modeling of non-crystalline (amorphous) substances, e.g.  ferrofluids and amorphous magnets, see e.g. \cite{Romano},\cite[Section 11]{OHandley},\cite{Bov} and \cite{DKKP,DKKP1}.
Each particle is characterized by its position $x\in\mathbb{R}^d$ and an internal parameter (spin)
$q _{x}\in \mathbb{R}^{\nu}$. For a given fixed (``quenched'') configuration $\gamma $ of particle positions, which is a locally finite
subset of $\mathbb{R}^{d}$, one~considers a system of differential
equations describing (non-equilibrium) dynamics of spins $q_{x}$,
$x\in \gamma $.

Two spins $q_{x}$ and $q_{y}$ are allowed to
interact via a pair potential if the distance between $x$ and $y$ is no more
than a fixed interaction radius $r>0$, that is, they are neighbors in the
geometric graph defined by $\gamma $ and $r$. The case where the vertex degrees of the graph are globally bounded (in particular, if $\gamma$ has a regular structure, e.g. $\gamma=\mathbb{Z}^d$) has been well-studied (even in the stochastic case), see \cite{LLL}, \cite{DZ1} and modern developments in \cite{INZ}, and
references therein. However, aforementioned applications to non-crystalline substances require to deal with unbounded vertex degrees. For example, if  configuration $\gamma$ is distributed according to a Poisson or, more generally, Gibbs measure, the typical number of ``neighbors'' of a particle located at  $x\in\mathbb{R}^d$ is proportional to $\log |x|$ (see Example~\ref{ex:max} below).

More generally, we consider, for a fixed $\gamma\subset\mathbb{R}^d$, the system
  \begin{equation}
\dt{q}_{x}(t)=F_{x}(\bar{q}(t)), \quad
q_{x}(0)=q_{x,0}\in S, \quad x\in \gamma,
\label{ODE}
\end{equation}
where $\bar{q}(t)=\left( q_{x}(t)\right) _{x\in \gamma }\in S^\gamma$,  $t\in\R_+:=[0,\infty)$ and $F_x:S^\gamma\to S$ for each $x\in \gamma$.
 We suppose that the system is row-finite, that is, $F_x$ depends only on
 a~finite number $n_x$ of components of the vector $\bar{q}$,
 which may be unbounded in $x\in\gamma$.

 A natural approach to the study of  system \eqref{ODE} would be to realize it as an evolution equation (with ``good'' coefficients) in a Banach space of sequences. However, in many cases, this proves to be impossible even for linear systems, as it happens for instance on unbounded vertex degree graphs. Beyond the framework of a fixed Banach space, linear row-finite systems  are always solvable, although the corresponding solutions might show an uncontrolled growth in $x\in\gamma$, see e.g. in \cite[\S\,6]{Deim}.

The aim of this work is to show the existence of solutions $q_{x}(t)$, $x\in\gamma$,  to system \eqref{ODE} keeping a control over the growth of the solution in $x\in\gamma$. In order to do this, we introduce a suitable increasing scale of Banach spaces $S_{\alpha }^{\gamma}$, $\alpha\ge 0$.  Our main result (Theorem~\ref{theor-main}) states that for
 each $0<\alpha <\beta $  and  any initial condition $ \bar{q}_{0} \in S_{\alpha }^{\gamma }$
there exists a  solution $\bar{q}(t)\in S_{\beta }^\gamma$  of system \eqref{ODE} on infinite time-interval $[0,\infty)$, subject to certain dissipativity-type conditions on $F_x$ and a bound on the order of growth of $n_x$ in $x\in\gamma$.

For the proof, we approximate \eqref{ODE} by finite systems. Uniform estimates of the corresponding solutions are obtained using a version of the so-called Ovsyannikov method for linear systems in a scale of Banach spaces. In contrast to the classical Ovsyannikov method (see e.g. \cite{Deim}), our
 modified version gives the existence of solutions with infinite lifetime (Theorem~\ref{Th-ovs}).

 We also prove the uniqueness of this solution, which requires however stronger conditions on $F_x$ (Theorem~\ref{theor-uniq}).

 We present three types of examples. In Subsection~\ref{gtd}, we consider a general gradient-type system of the form \eqref{ODE}. In~Subsection~\ref{ias}, we replace $(q_x)_{x\in\gamma}$ in the right hand side of \eqref{ODE} by the pairs $(p_x,q_x)_{x\in\gamma}$ with $\frac{d}{dt}p_x=q_x$ to study an infinite anharmonic (Hamiltonian) systems. Note that the latter example, in the case of the regular $\gamma=\mathbb{Z}^d$, was studied in \cite{LLL}.
Finally, in Subsection~\ref{sad}, we consider an example of multiparticle dynamics of somewhat different type, motivated by the study of self-organized systems, see review in \cite{MT}.

\section{Problem description and the existence result\label{PDER}}

\subsection{The setup and main result}
We start with more precise description of our model.
 We suppose that $\gamma$ is a locally finite subset of the space $X=\R^d$, $d\geq1$, i.e.  the set $\gamma\cap\Lambda$ is finite for any compact $\Lambda\subset X$. We denote by $S^\eta$ the vector space of elements of the form $(q_x)_{x\in\eta}$ with $q_x\in S$ for $x\in\eta\subseteq \gamma$.

Let us fix a number $r>0$ and introduce the  family of  \emph{finite} sets
\[
	\gamma _{x,r} :=\bigl\{ y\in \gamma \bigm\vert \left\vert x-y\right\vert \leq
r\bigr\}, \quad x\in \gamma.
\]
Consider the space  $S^{\gamma_{x,r}}$ endowed with the (finite) Cartesian product topology and introduce the notation
$\bar{q}_{x,r}:=\left( q_{u}\right) _{u\in \gamma_{x,r}}$.

We suppose that the following condition holds, which reflects the fact that the system \eqref{ODE} is row-finite.
\begin{condition}\label{cond:row-finite}
For each $x\in\gamma$ there exists $f_x\in C^1(S^{\gamma_{x,r}},S)$ such that
		$$F_{x}(\bar{q})=f_{x}\bigl(\bar{q}_{x,r}\bigr), 	 \ \bar{q}\in S^\gamma.$$
\end{condition}

Throughout the paper, we will understand solutions of \eqref{ODE} in the following sense.
\begin{definition}
We call a map $\bar{q}:[0,T)\to S^\gamma$ a (pointwise) solution of system \eqref{ODE} if the function $t\mapsto q_x(t)\in S$ is continuous (resp. continuously differentiable) on the interval $[0,T)$ (resp. $(0,T)$) and satisfies \eqref{ODE}  for each $x\in\gamma$.
\end{definition}

In what follows, we assume that the family of mappings $(F_x)_{x\in \gamma}$ is in certain sense dissipative. To this end, let  $U_x\in C^2(S,\R_+)$, $x\in \gamma$, be a family of functions such that
\begin{alignat}{2}
U_{x}(q)&\geq C_{1}\left\vert q\right\vert, &&\qquad q\in S, \ |q|\geq 1,  \label{c0}\\
U_{x}(q)&\leq C_{2}(\left\vert q\right\vert^j+1),   &&\qquad q\in S, \label{c01}
\end{alignat}
for some $C_1, C_2>0$ and $j\in\N$.
Next, we introduce the following notation:
\[
\text{for }  x\in \gamma, \,\,y\sim x \text{ means that } y\in \gamma _{x,r}.
\]
Let also
$n_{x}= n_{x,r}(\gamma )\geq1$
denote the number of points in $\gamma _{x,r}$, $x\in\gamma$.
\begin{condition}\label{cond:diss}
Let $U_x\in C^2(S,\R_+)$, $x\in \gamma$, satisfy \eqref{c0}--\eqref{c01}, and
there exist $C>0$ and $m\in\N$ such that
\begin{equation}\label{eq:conddiss}
		F_{x}(\bar{q})\cdot \nabla U_{x}(q_x)\leq C \sum_{y\sim x} (n_x n_y)^m U_y(q_y), \quad \bar{q}\in S^\gamma,\ x\in\gamma.
\end{equation}
\end{condition}
Here the dot $\cdot$ denotes Euclidean inner product in $S$.
Examples of  the families  $(F_x)_{x\in \gamma}$  satisfying Conditions~\ref{cond:row-finite}--\ref{cond:diss} will be given in Section~\ref{IAS}.

\smallskip
Let $w:\R_+\to[1,\infty)$ be a non-decreasing  function.  We define the  family of Banach spaces
\begin{equation}
S_{\alpha }^{\gamma }:=\Bigl\{ \bar{q} \in S^\gamma \Bigm\vert \lVert \bar{q} \rVert _{\alpha }:=\sup_{x\in \gamma }\frac{\left\lvert
q_{x}\right\rvert}{w(|x|)^{\alpha}}<\infty \Bigr\} ,\quad
\alpha >0.\label{defSalphagamma}
\end{equation}
Here and below, with an abuse of notations, $|x|$ means the Euclidean norm in $X=\R^d$, whereas $|q_x|$ means the Euclidean norm in $S=\R^\nu$. Clearly, the family \eqref{defSalphagamma} is increasing in $\alpha$, i.e. $S_\alpha^\gamma\subset S_\beta^\gamma$ for $0<\alpha <\beta$.

We introduce now certain balance condition on the growth of $n_x$ and $w(|x|)$ as $|x|\rightarrow\infty$.
\begin{definition}
Let $\mathcal{R}$ denote the class of non-decreasing functions $f:\R_+\to[1,\infty)$ such that
\begin{equation}\label{eq:condw}
		f_\tau:=\sup_{s\in\R_+}\frac{f(\tau+s)}{f(s)}<\infty
\end{equation}
for any $\tau>0$.
\end{definition}

Examples of functions from $\mathcal{R}$ are considered in Subsection~\ref{subsec:disc-ex} below.

\begin{definition}
Let  $w,z\in\mathcal{R}$. We call the pair $(w,z)$ admissible if for any $\beta,\mu>0$ and $\alpha\in (0,\beta)$ there exists $D=D(\beta,\mu)\geq 1$ such that
\begin{equation}\label{eq:balancecond}
		\sup_{s\in\R_+} z(s)^{\mu} w(s)^{-\alpha}\leq \frac{D}{\alpha}.
\end{equation}
\end{definition}

We formulate now the balance condition. Let $w$ be the weight function defining the scale of spaces $S^\gamma_\alpha$ in \eqref{defSalphagamma}.
\begin{condition}\label{cond:balance}
There exists $z\in\mathcal R$ such that
\begin{equation}\label{eq:nx}
		n_x\leq z(|x|), \quad x\in\gamma,
\end{equation}
and the pair  $(w,z)$ is admissible.
\end{condition}

The following  is our main existence result.
\begin{theorem}
\label{theor-main}
  Let Conditions~\ref{cond:row-finite}--\ref{cond:balance} hold.
   Then, for
 each $\alpha\ge 0$  and any initial condition $ \bar{q}_{0}=(q_{x,0})_{x\in\gamma} \in S_{\alpha }^{\gamma }$,
there exists a pointwise solution $\bar{q}(t)$ of system \eqref{ODE} on infinite time-interval $[0,\infty)$. Moreover, for   $j\in\N$ from \eqref{c01} we have the inclusion
\[
	\bar{q}(t)\in \bigcap_{\beta>j \alpha} S_{\beta }^\gamma,\quad t\in [0,\infty),
\]
and the estimate
\begin{equation}
\lVert \bar{q}(t)\rVert_{\beta}\leq C_p(\alpha ,\beta ;t)\left(\lVert \bar{q}_0\rVert_\alpha^j+1\right) \label{q-bound}
\end{equation}
 holds true for any $\beta>j\alpha$ and $p>1$.
 Here  $C_p(\alpha ,\beta ):\R_+\to\R_+$  is  an entire function of order
$\rho=\frac{p}{p-1}>1$ and type $\sigma=B^\rho e^{\rho(\beta r+\beta-j\alpha+1)} (e\rho)^{-1}(\beta- j\alpha)^{- \frac{1}{p-1}}>0$
with some $B=B(p)>0$.
\end{theorem}
\begin{remark}\label{rem:largesigma}
Clearly, by choosing $p>1$ large enough, one can make the order $\rho>1$ arbitrary close to $1$.
\end{remark}
\begin{remark}The order and type of $C_p$ are positive and finite. Thus, for any $\varepsilon>0$, there exists $T_\varepsilon>0$ such that
\[
\sup_{t\in[0,T]}\lVert \bar{q}(t)\rVert_{\beta} \leq e^{(\sigma+\varepsilon)T^{\rho+\varepsilon}}
(\lVert \bar{q}_0\rVert_{\alpha}^j+1), \quad
T>T_\varepsilon.
\]
\end{remark}

\subsection{Discussion of Condition~\ref{cond:balance} and examples}\label{subsec:disc-ex}

Our next goal is to explain in more detail the definition of admissible pairs.

\medskip
We start with the discussion of the class $\mathcal R$.
Let $f:\R_+\to[1,\infty)$ be a non-decreasing function. Then, for
each $\tau>0$, the function $\frac{f(\tau+s)}{f(s)}$ is bounded in $s$ on any closed subinterval of $\R_+$. Thus a sufficient condition for \eqref{eq:condw} is that $\lim\limits_{s\to\infty} \frac{f(\tau+s)}{f(s)}<\infty$, $\tau>0$. The latter, by e.g. \cite[Theorem~1.4.1]{BGT1987} applied to the function $g(s):=f(\log s)$, $s>1$, is equivalent to the requirment that $g$ has regular variation of index $\rho\in\R$, i.e. $\lim\limits_{s\to\infty} \frac{g(\tau s)}{g(s)}=\tau^\rho$, $\tau>0$. Combining \cite[Theorem~1.4.1]{BGT1987} and the examples of \cite[p.~16]{BGT1987}, we conclude that a function $g(s)=s^\rho h(s)$, $\rho\in\R$, has regular variation of index $\rho$, if e.g. $h(s)$ is a product of non-negative powers of $\log s$, $\log\log s$ (and so on), $e^{(\log s)^\nu}$, $\nu\in(0,1)$ etc.
Considering $f(s):=g(e^s)$, $s\in\R_+$, we get that e.g. the functions
\[
	f(s)= (\log s)^\kappa s^\nu e^{\rho s^\mu}, \quad \rho, \kappa,\nu\geq0, \ \mu\in(0,1]
\]
belong to the set $\mathcal{R}$, as well as the functions $\log f(s)$, $\log\log f(s)$ and so on.

\medskip
The following Lemma describes a simple way to generate admissible pairs.

\begin{lemma}\label{rem:doublelog}
\begin{enumerate}[label=(\arabic*)]
\item Let $w\in\mathcal{R}$, $w\geq e^e$, and  $z$ be given by the formula
\[
	z(s)=\upsilon \log(\log w(s)),  \quad \upsilon >0.
\]
Then the pair $(w,z)$  is admissible.
\item Let $(w,z)$  be an admissible pair. Then the pair $(w+c_1,z+c_2)$ is admissible for any $c_1,c_2\ge 0$.
\item
Let $(w,z)$  be an admissible pair, $v\in\mathcal{R}$ and $v(s)\geq w(s)$, $s\in\R_+$. Then $(v,z)$ is an admissible pair.
\end{enumerate}
\end{lemma}
\begin{proof} As we already pointed out,   $z\in\mathcal R$.
It is straightforward to check that, for any $\mu>0$, there exists $d_\mu>0$, such that $(\log \tau)^\mu \leq d_\mu \tau$ for all $\tau>0$. Then, denoting $v(s):=\log w(s)\geq 0$, $s\in\R_+$, we obtain
\[
	z(s)^{\mu} w(s)^{-\alpha}\leq \upsilon^\mu d_\mu  \bigl(\log w(s)\bigr) w(s)^{-\alpha}=\upsilon^\mu d_\mu v(s) e^{-\alpha v(s)}\leq \frac{\upsilon^\mu d_\mu}{e \alpha},
\]
which completes the proof of part (1).
Parts (2) and (3) are  obvious.
\end{proof}

\begin{remark}\label{rem:obviousbut}
In (1), it is sufficient to assume that $w(s)\geq e^e$, $s\geq s_0$ for some $s_0>0$ only, and set e.g. $w(s)=1$ for $s\in[0,s_0)$. Then one can choose $z(s)=\upsilon \log(\log(w(s)))$ for $s\geq s_0$ and $z(s)=1$ otherwise.
\end{remark}

Let us note that the  structure of the given underlying set $\gamma$ dictates the choice of the function $z$, which in turn determines suitable weight function $w$ and, ultimately, the conditions on the family $F_x,\,x\in\gamma$. Below are three examples of admissible pairs associated with different type of $\gamma$.
\begin{example} [Minimal growth] \label{ex:min}
Assume that the number of elements in $\gamma _{x,r}$ is globally bounded,
that is, there exists a constant $z_{0}\geq1$ such that $n_{x}\leq z_{0}$ for all $x\in \gamma $. An important example of such $\gamma$ is given by the integer lattice  $\mathbb{Z}^{d}$.
 Then we can set $z(s)\equiv z_{0}$ and
choose \emph{an arbitrary} non-decreasing function $w:\R_+\to[1,\infty)$ such that $w(s)\geq e^e$ for $s\geq s_0$, cf. Remark~\ref{rem:obviousbut}. The choice of $w$ is dictated by the growth of the initial condition $\bar{q}_{0}=(q_{x,0})_{x\in\gamma}\in S_{\alpha }^{\gamma }$, cf.~Theorem~\ref{theor-main} (i.e. by the growth of $|q_{x,0}|$ in $x\in\gamma$).
Note also that our
 main technical tool, the modified Ovsyannikov theorem (cf. Theorem~\ref{Th-ovs} and Proposition~\ref{theLemma}), becomes
then  redundant, as the corresponding linear operator $A$ defined by formula \eqref{eq:defA} will be bounded in any such (fixed) $S_{\alpha }^{\gamma }$, $\alpha>0$ (i.e. the dynamics evolves in one Banach space).
This case is well-studied, see \cite{LLL},
where $w(s)=e^s$ (and $\alpha=1$) was considered.
In~particular, if $ (|q_{x,0}|)_{x\in\gamma}$ is bounded, one can choose $w(s)\equiv e^e$; then all the spaces $S^\gamma_\alpha$, $\alpha>0$ will
coincide with the space $l_\infty(S)$ of bounded sequences (and $A$ is bounded there).
\end{example}

\begin{example}[Maximal growth]\label{ex:max}
Assume that the number of elements in $\gamma _{x,r}$ grows logarithmically,
that is, there exists $a(r)>0$ such that
\begin{equation}\label{nz}
n_{x,r}(\gamma )\leq z(\left\vert x\right\vert),\qquad z(s):=
a(r) \left(1+\log (1+s)\right)
\end{equation}
for all $x\in X$ and $r>0$. This bound holds for a typical random
configuration distributed according to a Ruelle measure on $X$, see e.g.
\cite{R70} and \cite[p.~1047]{K93}.
Then we can set $w(s)=e^s$, $s\geq e$, cf. Remark~\ref{rem:obviousbut}. It follows from Lemma~\ref{rem:doublelog} that the pair $(w,z)$ is admissible. This is our most important motivating case, see Section~\ref{IAS} and Example~\ref{ex:maxrev}.
\end{example}

\begin{example}[Medium growth]\label{ex:med}
Let,
for some $\upsilon >0$,
\begin{equation}\label{mednz}
		z(s)=\upsilon \left( 1+\log \log (e+s)\right).
\end{equation}
Then, according to Lemma~\ref{rem:doublelog}, we can
set $w(s)=1+s$.  In this case, more comprehensive study of the solutions of system \eqref{ODE} can be accomplished, at least in the framework of systems with pair interaction, which we discuss in Section~\ref{IAS} (see Example~\ref{ex:medrev}). Indeed, slow growth of the weight function $w$ allows to show the uniqueness of the solution living in $S^\gamma_\beta$ with any (fixed) $\beta>0$.
\end{example}

\subsection{Scheme of the proof}

To explain the scheme of the proof of Theorem~\ref{theor-main}, we start with the following simple observation. Let $\bar{q}(t)$ solve \eqref{ODE}, and define
\begin{equation}\label{eq:L}
	\L_x(t):=U_x(q_x(t)), \quad x\in\gamma,\ t\geq0.
\end{equation}
Then Condition~\ref{cond:diss} yields
\begin{align}
\dt\L_x(t)&=\dt q_x(t)\cdot \nabla U_x(q_x(t))=F_x(\bar{q}(t))\cdot \nabla U_x(q_x(t))\notag\\
&\leq C \sum_{y\sim x} (n_x n_y)^m U_y(q_y)=\sum_{y\in\gamma} A_{x,y} \L_y(t),\label{eq:major}
\end{align}
where
\begin{equation}\label{eq:defA}
A_{x,y}:=\begin{cases}
						C (n_x n_y)^m, & y\sim x,\\
						0, &\text{otherwise}.
					\end{cases}
\end{equation}
Introduce the infinite matrix $A=(A_{x,y})_{x,y\in\gamma}$. Relation \eqref{eq:major} implies  the following differential inequality in $S^\gamma$:
\begin{equation}\label{eq:diffineq}
	\dt\L(t)\leq A \L(t), \quad \L(t)=(\L_x(t))_{x\in\gamma}.
\end{equation}
Now we can proceed as follows. First we will prove that, for any $0<\alpha<\beta$, the differential equation
\begin{equation}\label{eq:dtPsi}
\dt \Psi(t)=A\Psi(t) ,\quad \Psi (0)\in S_\alpha^\gamma,
\end{equation}
which corresponds to the inequality \eqref{eq:diffineq}, has a  classical  solution with infinite lifetime in  $S_\beta^\gamma$ (Proposition~\ref{theLemma}). To this end, we use a modified version of Ovsyannikov's method (proved in Theorem~\ref{Th-ovs}).
Informally, we will have then that $\L_x(t)\leq \Psi_x(t)$, $x\in\gamma$, $t\geq0$ (here and below we understand inequalities between elements of $S=\R^\nu$ in the coordinate-wise sense).
Next, we will  approximate \eqref{ODE} by finite volume cut-off systems and use the corresponding finite-dimensional versions of \eqref{eq:diffineq} and \eqref{eq:dtPsi} to find (uniform in volume $\Lambda\subset X$) estimates, which will allow us to pass to a limit  as $\Lambda \rightarrow X$.

\section{Linear equations in a scale of Banach spaces}

In this section we prove a general result on the existence of (infinite-time) solutions for a special class of  linear differential equations, which extends the so-called Ovsyannikov method, see e.g. \cite{Deim}.

Let us consider a family of Banach spaces $\B_{\alpha }$ indexed
by $\alpha \in   (  0,\beta ] $ with $\beta <\infty $ fixed, and
denote by $\left\Vert \cdot \right\Vert _{\alpha }$ the corresponding norms.
We assume that
\begin{equation}
\B_{\alpha ^{\prime }}\subset \B_{\alpha ^{\prime \prime }}\ \text{and }
\left\Vert u\right\Vert _{\alpha ^{\prime \prime }}\leq \left\Vert
u\right\Vert _{\alpha ^{\prime }}\text{ if }\alpha ^{\prime }<\alpha
^{\prime \prime },\ u\in \B_{\alpha ^{\prime \prime }},  \label{scale}
\end{equation}
where the embedding means that $\B_{\alpha ^{\prime }}$ is a vector subspace
of $\B_{\alpha ^{\prime \prime }}$. For any $\delta\in(0,\beta]$, we set
$\B_{\delta-}:=\bigcup\limits_{0<\alpha <\delta }\B_{\alpha }$.

Let $A:\B_{\beta-}\rightarrow \B_{\beta-}$ be a linear operator. We assume that $A$ is a bounded operator acting from $\B_{\alpha'}$ to $\B_{\alpha''}$ for any $0< \alpha'<\alpha'' \leq\beta  $, and
\begin{equation}
\left\Vert Ax\right\Vert _{\alpha ^{\prime \prime }}\leq c\left(   \alpha''- \alpha'  \right) ^{-q}\left\Vert x\right\Vert _{\alpha ^{\prime }}
\label{A-cond}
\end{equation}
for all $x\in \B_{\alpha ^{\prime }}$ and some constants $c=c(\beta)>0$ and $q\in (0,1)
$ (both independent of $\alpha ^{\prime }$ and $\alpha ^{\prime \prime }$).

\begin{theorem}
\label{Th-ovs} Assume that the bound (\ref{A-cond}) holds. Then, for any $  \alpha \in (0,\beta) $
and $u_{0}\in \B_{\alpha }$, there exists   a continuous function $
u:[0,\infty )\rightarrow \B_{\beta }$ with $u(0)=u_0$  such that:
\begin{enumerate}[label=\arabic*)]
\item $u$ is continuously differentiable on $(0,\infty )$;
\item $Au(t)\in \B_{\beta }$ for all $t\in (0,\infty )$;
\item $u$ solves the differential equation
\begin{equation}
\dfrac{d}{dt}u(t)=Au(t),\quad t\in (0,\infty ).  \label{IVP}
\end{equation}
\end{enumerate}
  Moreover,
\begin{equation}
\lVert u(t)\rVert_\beta\leq a(t) \lVert u_0\rVert_\alpha, \quad t>0,
\end{equation}
where $a(t)$ is an entire function of order $\rho=\frac{1}{1-q}$ and type $\sigma=(ce)^\rho (e\rho)^{-1}(\beta- \alpha)^{-q \rho}$.
\end{theorem}

\begin{proof}
Let us first observe that, by (\ref{A-cond}), $A$ is a well-defined operator in any $\B_{\delta -}$, $\delta \in (\alpha ,\beta )$.
 Thus, for any $n\geq 1$,   one can consider  $A^{n}:\B_{\delta -}\rightarrow \B_{\delta -}$. Embeddings $\B_{\alpha }\subset
\B_{\delta -}\subset \B_{\delta }$ imply that $A^{n}:\B_{\alpha }\rightarrow
\B_{\delta }$ is a well-defined   linear  operator.

Fix $u_{0}\in \B_{\alpha }$ and define the sequence $u_k=A^ku_0\,(=Au_{k-1})$, $k=1,2,...$. Then,
for any $n\ge 1$ and $\alpha _{1},...,\alpha _{n}$ such that $\alpha <\alpha
_{1}<...<\alpha _{n}=\delta $, we have  $ u_{k} \in \B_{\alpha_k-}\subset   \B_{\alpha _{k}}$ and
\[
\left\Vert u_{k}\right\Vert _{\alpha _{k}}\leq c\left( \alpha _{k}-\alpha
_{k-1}\right) ^{-q}\left\Vert u_{k-1}\right\Vert _{\alpha _{k-1}}.
\]
Setting $\alpha _{k}:=\alpha +k\frac{\delta -\alpha }{n}$, we obtain the estimate
\begin{equation*}
\left\Vert A^{n}u_{0}\right\Vert _{\delta }\leq D^{n}n^{qn}\left\Vert
u_{0}\right\Vert _{\alpha },\quad D:=c(\delta -\alpha )^{-q}, \quad   n\geq 1.
\end{equation*}
  Therefore the series
\begin{equation}
u(t):=\sum_{n=0}^{\infty }\frac{t^{n}}{n!}A^{n}u_0  \label{u-t}
\end{equation}
is uniformly convergent in $\B_\delta$ on the interval $[0,T]$  for each $T>0$. Indeed,  the inequality $n!\geq \bigl(\frac{n}{e}\bigr)^n$ implies that
\begin{equation}\label{entire}
	 \lVert u(t)\rVert_\delta\leq \lVert u_0\rVert_\alpha\sum_{n=0}^{\infty } \frac{(De)^n}{n^{(1-q)n}}t^n<\infty
\end{equation}
for any $t>0$, and the series in the r.h.s. of \eqref{entire} defines an entire function.

Similarly, the series
\begin{equation*}
v(t):=\sum_{n=0}^{\infty }\frac{t^{n}}{n!}A^{n+1}u_0
\end{equation*}
is uniformly convergent in $\B_\delta$ on the interval $[0,T]$. As a result,
\begin{equation}\label{visdtu}
v(t)=\frac{d}{dt}u(t)
\end{equation}
in $\B_{\delta }$. Since the norm in $\B_\delta$ is stronger than
in $\B_\beta$, the function $[0,T]\ni t\mapsto u(t)$ is differentiable in  $\B_\beta$ and \eqref{visdtu} holds in $\B_\beta$, too.
Observe that $A$ maps $\B_{\delta }\to \B_{\beta }$ continuously.  Therefore, we can apply $A$ to the right hand side of \eqref{u-t} and obtain $Au(t)=v(t)\in \B_{\beta }$. This proves parts 1) -- 3) of the statement.

Next, fix any $\delta_0\in(\alpha,\beta)$. The series in the r.h.s. of \eqref{entire} converges uniformly in $\delta\in[\delta_0,\beta]$ (recall that $D$ depends on $\delta$) to a continuous decreasing function of $\delta$. Therefore, one can pass to the limit as $\delta\to\beta$ in \eqref{entire}, to get
\begin{equation}\label{entire1}
	\lVert u(t)\rVert_\beta\leq \lVert u_0\rVert_\alpha\sum_{n=0}^{\infty } \frac{(Be)^n}{n^{(1-q)n}}t^n<\infty, \quad B:=c(\beta -\alpha )^{-q}.
\end{equation}
The sum of the series in the r.h.s. of \eqref{entire1} is an entire function of order
\[
\rho:=\limsup_{n\to\infty}\frac{n\ln n}{-\ln\frac{(Be)^n}{n^{(1-q)n}}}=\frac{1}{1-q}.
\]
The type $\sigma $ of this entire function satisfies the equality
\[
(\sigma e\rho)^{\frac{1}{\rho}}=\limsup_{n\to\infty}n^{\frac{1}{\rho}}\biggl(\frac{(Be)^n}{n^{(1-q)n}}\biggr)^{\frac{1}{n}} =Be,
\]
i.e.
\[
	\sigma=\frac{(Be)^\rho}{e\rho}=\frac{(ce)^\rho}{e\rho(\beta- \alpha)^{q \rho}},
\]
which completes the proof.
\end{proof}

\begin{remark}
The order and type of $a(t)$ are positive and finite. Thus, for any $\varepsilon>0$, there exists $T_\varepsilon>0$ such that
\[
\sup_{t\in[0,T]}\lVert u(t)\rVert_\beta \leq e^{(\sigma+\varepsilon)T^{\rho+\varepsilon}}\lVert u_0\rVert_{\alpha}
\]
for all $T>T_\varepsilon$.
\end{remark}

\begin{remark}
The function $u$ from Theorem~\ref{Th-ovs} is unique. The proof can be obtained in a similar way to the ``standard'' Ovsyannikov method, see e.g. \cite[Theorem~2.4]{Fink}. Note that we do not use this fact in the present paper.
\end{remark}

\section{Proof of the existence result}
\subsection{Row-finite linear systems}

We start with the study of the linear system \eqref{eq:dtPsi}.  First, we show that the matrix  \eqref{eq:defA} generates a bounded operator in the scale $(S_\alpha^\gamma)_{\alpha>0}$.
\begin{lemma} \label{lemma-A}
Let Condition~\ref{cond:balance} hold.
Then the matrix $\left( A_{x,y}\right) _{x,y\in \gamma }$, given by \eqref{eq:defA}, generates a
bounded   linear   operator $A:S_{\alpha' }^{\gamma }\rightarrow S_{\alpha'' }^{\gamma }$
for any $0<\alpha' <\alpha'' <\beta$.   The corresponding norm of $A$  satisfies the estimate
\begin{equation}
\left\Vert A\right\Vert _{\alpha', \alpha'' }\leq   B w_r^{\alpha'}  \left( \alpha''-\alpha' \right) ^{-1/p}\label{eq:est21}
\end{equation}
with an arbitrary $p>1$ and some $B=B(p)>0$.
\end{lemma}

\begin{proof}
Fix arbitrary $0<\alpha' <\alpha'' $.
For any $\bar{q}\in S_{\alpha' }^{\gamma}$, we have
\begin{equation*}
\left\Vert A\bar{q}\right\Vert _{\alpha'' }\leq C \sup_{x\in \gamma } \sum_{y\sim x}(n_x n_y)^m\lvert q_{y}\rvert \,
w(|x|)^{-(\alpha''-\alpha' )}w(|x|)^{-\alpha'},
\end{equation*}
by \eqref{defSalphagamma}, \eqref{eq:defA}.
Recall that $y\sim x$ implies $\left\vert y-x\right\vert \leq r$, and hence
\begin{equation}\label{eq:simple}
\lvert y\rvert \leq r+\lvert x\rvert.
\end{equation}
Then, since $w$ is  and satisfies \eqref{eq:condw}, we have
\[
	w(|x|)^{-\alpha'}\leq \biggl(\frac{w(r+\lvert x\rvert)}{w(|x|)}\biggr)^{\alpha'}\,w(|y|)^{-\alpha'}\leq w_r^{\alpha'} w(|y|)^{-\alpha'},
\]
where $w_r$ is as in \eqref{eq:condw},  so that
\begin{align}
\left\Vert A\bar{q}\right\Vert _{\alpha'' }&\leq Cw_r^{\alpha'}\sup_{x\in \gamma
}w(|x|)^{-(\alpha''-\alpha' )}\sum_{y\sim x}(n_x n_y)^m\lvert q_{y}\rvert w(|y|)^{-\alpha'}\notag \\& \leq w_r^{\alpha'}\lVert \bar{q} \rVert _{\alpha' }\sup_{x\in \gamma
}\left( A_{x}w(|x|)^{-(\alpha''-\alpha' )}\right) ,\label{eq:est1}
\end{align}
where
\begin{equation}
	A_{x}:=C\sum\limits_{y\sim x }(n_x n_y)^m, \quad x\in\gamma. \label{eq:est2}
\end{equation}
By \eqref{eq:nx}, \eqref{eq:simple} and \eqref{eq:condw}, we have
\[
	n_y\leq z(|y|)\leq \frac{z(r+|x|)}{z(x)} z(|x|)\leq z_r z(|x|), \quad y\sim x, \ x\in\gamma.
\]
Therefore,
\[
	A_x\leq C n_x^{m+1} z_r^m z(|x|)^m =M z(|x|)^{2m+1},
\]
where $M:=C z_r^m$.

Fix an arbitrary $p>1$.
By \eqref{eq:est1} and \eqref{eq:balancecond} we obtain the bound
\begin{align*}
\left\Vert A\bar{q}\right\Vert _{\alpha'' }&\leq  M w_r^{\alpha'}\lVert \bar{q} \rVert _{\alpha' }\sup_{x\in \gamma
}\left( z(|x|)^{p(2m+1)}w(|x|)^{-p(\alpha''-\alpha' )}\right)^{\frac{1}{p}}\\
&\leq  M w_r^{\alpha'}\lVert \bar{q} \rVert _{\alpha' }\left( \frac{D\bigl(p\beta,p(2m+1)\bigr)}{p(\alpha''-\alpha' )}\right)^{\frac{1}{p}}=B w_r^{\alpha'} \left( \alpha''-\alpha' \right) ^{-1/p}\lVert \bar{q} \rVert _{\alpha' },
\end{align*}
where $B=M \bigl(\frac{1}{p}D(p\beta,p(2m+1))\bigr)^{\frac{1}{p}}>0$.
 This completes the proof.
\end{proof}

Now we can use Theorem~\ref{Th-ovs} and prove the existence of solutions of equation \eqref{eq:dtPsi}.
\begin{proposition}\label{theLemma}
   Let Condition~\ref{cond:balance} hold and let
 $A=(A_{x,y})_{x,y\in\gamma}$ be given by \eqref{eq:defA}. Then, for any  $0<\alpha <\beta $,  equation \eqref{eq:dtPsi} with $\Psi(0)\in S_{\alpha}^{\gamma}$
has a  classical  solution $\Psi (t)\in S_{\beta }^{\gamma}$, $t\geq 0$.   Moreover, the  estimate
\begin{equation}
\left\Vert \Psi (t)\right\Vert _{\beta }\leq A_p(\alpha ,\beta
;t)\left\Vert \Psi (0)\right\Vert _{\alpha },  \quad t\geq0, \label{psi-est}
\end{equation}
 holds true for each $p>1$, with an entire function  $A_p(\alpha ,\beta;\cdot):\R_+\to\R_+$ of order
$\rho=\frac{p}{p-1}>1$ and
type $\sigma=B^\rho  w_r^{\beta\rho} e^{\rho-1} \rho^{-1}(\beta- \alpha)^{- \frac{1}{p-1}}>0$.
\end{proposition}
\begin{proof}
Take any $p>1$. By Lemma~\ref{lemma-A}, for any $0<\alpha<\alpha'<\alpha''\leq \beta$, the estimate \eqref{A-cond} (for the scale of spaces $\B_\alpha=S_\alpha^\gamma$) holds with $q=\frac{1}{p}\in(0,1)$ and $c=B e^{\beta r+\beta-\alpha}>B e^{\alpha' r+\alpha''-\alpha'}>0$. Then the result follows directly from Theorem~\ref{Th-ovs}.
\end{proof}

\subsection{Finite-dimensional approximations}

For a compact set $\Lambda \subset X$, consider the collection $\bar{q}^\Lambda(t)=\left( q_{x}^\Lambda(t)\right) _{x\in \gamma }$ of functions $q_{x}^\Lambda:\R_+\to S$ such that
\begin{equation}
\begin{cases}
\dt q_{x}^{\Lambda }(t)=F_{x}(  \bar{q}^{\Lambda } (t)), & t>0, \ x\in \gamma_\Lambda,\\[2mm]
 q_{x}^{\Lambda }(t)=q_x(0), & t>0, \ x\in \gamma_{\Lambda^c},\\[2mm]
q_{x}^{\Lambda}(0)=q_x(0), & x\in\gamma,
\end{cases}  \label{ODE-fin}
\end{equation}
where $\gamma_\Lambda=\gamma\cap \Lambda$, $\gamma_{\Lambda^c}=\gamma\cap \Lambda^c$, $\Lambda^c:=X\setminus \Lambda$.

\begin{proposition}
\label{proposition-bound}
Let Conditions~\ref{cond:row-finite}--\ref{cond:balance} hold.
Then, for any compact $\Lambda \subset X$ and  $\alpha>0$, the system \eqref{ODE-fin} with an initial condition $\bar{q}_{0}=(q_x(0))_{x\in\gamma}\in S_{\alpha }^{\gamma }$ has a unique solution $\bar{q}^{\Lambda }(t)$, $t\geq 0$. Moreover, for each $x\in\gamma$, there exists
an entire function $Q_{x}:\R_+\to\R_+$, such that the estimate
\begin{equation}
\left\vert q_{x}^{\Lambda }(t)\right\vert \leq Q_{x}(t), \quad t\geq0,  \label{unif-est}
\end{equation}
holds for all compact sets $\Lambda\subset X $.
\end{proposition}
\begin{proof} For an arbitrary compact $\Lambda\subset X$, the system \eqref{ODE-fin} is essentially finite-dimensional with continuously differentiable coefficients. Thus there exists a unique solution of \eqref{ODE-fin} with lifetime $T_\Lambda\le\infty$. Moreover, $T_\Lambda<\infty$ implies that $|q_x^\Lambda(t)|\to\infty$ as $t\nearrow T_\Lambda$ for some $x\in\gamma_\Lambda$.

Similarly to \eqref{eq:L}, define
\begin{equation}\label{defofL}
	\L_x^\Lambda(t):=U_{x}(q_{x}^{\Lambda }(t)), \quad x\in\gamma, \ t\in[0,T_\Lambda),
\end{equation}
and note that
$\L_x^\Lambda(0)=\mathcal{L}_x(0)=U_{x}(q_{x}(0))$.
Observe that $\dt \L_x^\Lambda(t)=0$ for $x\in\gamma_{\Lambda_c}$, $t>0$. Thus, similarly to \eqref{eq:major}, we have the inequality
\begin{equation}
\dt\L_{x}^{\Lambda }(t)\leq \sum_{y\in \gamma }a_{x,y}^{  \Lambda }\mathcal{L}
_{y}^{\Lambda }(t),\quad x\in \gamma, \ t\in[0,T_\Lambda),  \label{L-bound2}
\end{equation}
where
\begin{equation}\label{eq:acutoff}
a^\Lambda_{x,y}:=\begin{cases}
				A_{x,y}, & x\in \Lambda, \\
				0,& \text{otherwise}
		\end{cases}
\end{equation}
and $A_{x,y}$ is given by \eqref{eq:defA}.
Denote
\[
\Lambda _{r} :=\bigl\{ x\in X\bigm\vert \mathrm{dist}\,\left( x,\Lambda \right) \leq r\bigr\}.
\]
We can replace $\gamma$ with $\gamma_{\Lambda_r}=\gamma\cap \Lambda_r$ in the r.h.s. of \eqref{eq:defA}, because $a^\Lambda_{x,y}=0$ if $y\not\in\Lambda _{r}$. Also, the relation $\frac{d}{dt}\L_{x}^{\Lambda }(t)=0\leq \mathcal{L}_{x}^{\Lambda }(t)$, $x\in\gamma\cap (\Lambda_r\setminus\Lambda)$, implies that \eqref{eq:defA}  always holds for $x\not\in\Lambda _{r}$. Thus the system of differential inequalities \eqref{L-bound2} is also essentially finite.

The classical comparison theorem (see e.g. \cite{Wal}) implies that
\begin{equation}\label{eq:addedlabel}
0\leq \mathcal{L}_{x}^{\Lambda }(t)\leq \Psi _{x}^{ \Lambda_r }(t), \quad   x\in\gamma_{\Lambda_r} ,
\end{equation}
where the collection of functions $(\Psi _{x}^{ \Lambda_r }(t))_{x\in \gamma_{\Lambda_r}}$ satisfies the following system of equations
\begin{equation*}
\dt{\Psi}_{x}^{ \Lambda_r }(t)=\sum_{y  \in\gamma_{\Lambda_r}  }a_{x,y}^{  \Lambda }\Psi
_{y}^{ \Lambda_r }(t),\quad \Psi _{x}^{ \Lambda_r }(0)=\mathcal{L}_{x}^{\Lambda }(0), \quad  x\in \gamma_{\Lambda_r}.
\end{equation*}
The latter system can be considered as a single equation with the cut-off matrix $A^\Lambda=(a_{x,y}^\Lambda)_{x,y\in\gamma}$:
\begin{equation}\label{eq:newsyst}
\dot{\Psi}^{ \Lambda_r }(t)=A^{  \Lambda  }\Psi^{ \Lambda_r } (t),\qquad \Psi_{x}^{ \Lambda_r }(0)=\mathcal{L}
_{x}^{\Lambda }(0), \  x\in\gamma,
\end{equation}
with the `trivial' extension $\Psi_{x}^{\Lambda_r }(t)
=\mathcal{L}_{x}^{\Lambda }(t)$,   $t\geq0$, for all   $x\in X\setminus \Lambda _{r}$.

 By \eqref {eq:L} and \eqref{c01} we have
 \begin{equation}\label{Lest1}
 \vert\L_{x}(0)\vert\leq C_2(|q_x(0)|^j+1).
 \end{equation}
  Therefore $\bar{q}_0\in S_\alpha^\gamma$ yields $\bigl(\mathcal{L}_{x}^{\Lambda }(0)\bigr)_{x\in \gamma}=\bigl(\mathcal{L}_{x}(0)\bigr)_{x\in \gamma}=:\mathcal{L}(0)\in S_{j\alpha}^\gamma$ and
  \begin{equation}
  \left\Vert \mathcal{L} (0)\right\Vert _{j\alpha }\leq  C_3\left( \Vert \bar{q_0}\Vert^j_\alpha +1\right)\label{L-est111}
  \end{equation}
  for some constant $C_3$.
   It is evident that  estimate \eqref{eq:est21} holds for the operator $A^\Lambda$ instead of $A$, with the same constant $B>0$. Therefore, using the arguments similar to those in the proof of Proposition~\ref{theLemma}, we can show that  the solution $\Psi ^{\Lambda_r }(t)$ to the system \eqref{eq:newsyst} exists for all $t\geq 0$ and satisfies the estimate
\begin{equation}
\left\Vert \Psi^{\Lambda_r } (t)\right\Vert _{j\beta }\leq A_p(j\alpha ,j\beta
;t)\left\Vert \mathcal{L} (0)\right\Vert _{j\alpha }, \quad\beta >\alpha, \quad t\geq0, \label{psi-est-La}
\end{equation}
with the same entire function $A_p$ as in \eqref{psi-est}.

Next, using \eqref{c0}, \eqref{defofL}, \eqref{eq:addedlabel} and \eqref{psi-est-La}, we can write,
for each $ t\geq0$,
\begin{align}
\left\vert q_{x}^{\Lambda }(t)\right\vert &\leq   \frac{1}{C_{1}}
U_{x}(q_{x}^{\Lambda }(t)) +1 =\frac{1}{C_{1}}\L
_{x}^{\Lambda }(t)+1 \leq \frac{1}{C_1}\Psi ^{\Lambda_r }_x(t)+1 \notag\\& \leq \frac{1}{C_1}A_p(j\alpha ,j\beta
;t)\left\Vert \mathcal{L} (0)\right\Vert _{j\alpha }
w(x)^{j\beta}+1=:Q_{x}(t)<\infty.\label{q-est111}
\end{align}
This bound holds for any $x\in\gamma$, which implies that $T_\Lambda=\infty$ and \eqref{unif-est} holds. The proof is complete.
\end{proof}

\begin{corollary}
Estimates (\ref{q-est111}) and (\ref{L-est111}) imply that
\begin{align}
\left\Vert \bar{q}^{\Lambda }(t)\right\Vert_{\beta}
 \leq
  C_p(\alpha ,\beta;t)\bigl(\left\Vert \bar{q}_0 \right\Vert^j _{\alpha } +1\bigr), \quad t\geq0,\label{q-est1111}
\end{align}
for any
$\beta>j\alpha$.
Here $C_p(\alpha ,\beta;\cdot):\R_+\to\R_+$  is an entire function of order
$\rho=\frac{p}{p-1}>1$ and type $\sigma=B^\rho e^{\rho(\beta r+\beta-j\alpha+1}) (e\rho)^{-1}(\beta- j\alpha)^{- \frac{1}{p-1}}>0$
with some $B=B(p)>0$.
\end{corollary}

Now we are ready to prove Theorem~\ref{theor-main}. The proof is essentially
similar to that in \cite{LLL}.

\begin{proof}[Proof of Theorem~\ref{theor-main}]

Fix
$\alpha>0$, $\beta>j \alpha$
and  $\bar{q}_{0}\in S_{\alpha }^{\gamma }$.  Choose any sequence ${\bf\Lambda}^0$ of compacts $\Lambda$ exhausting $X$.
Let $\bar{q}^{\Lambda}(t)$ solve the corresponding system \eqref{ODE-fin}. Observe that both sides of inequality \eqref{unif-est} are continuous in $t$, which implies that
\[
	\sup_{t\in[0,T]}\left\vert q_{x}^{\Lambda }(t)\right\vert\leq \sup_{t\in[0,T]}Q_{x}(t)<\infty,
\]
for each $x\in\gamma$ and any $T>0$. The equality $\dot{q}_{x}^{\Lambda }=F_{x}( \bar{q}^{\Lambda })$ together with continuity of $F_{x}$ on $S^{\gamma_{x,r}}$ imply then that $\sup_{t\in[0,T]}\left\vert \dot{q}_{x}^{\Lambda }(t)\right\vert$, $x\in\gamma$, is also bounded uniformly in $\Lambda$.

Let us fix an arbitrary indexation of $\gamma$, so that $\gamma=(x_k)_{k=1}^\infty$.
The Arzel\`a--Ascoli theorem implies that there exists a subsequence ${\bf\Lambda}^{(1)}=(\Lambda^{(1)}_{n})$ of ${\bf\Lambda}^{(0)}$ and some $q_{x_1}(t)\in S$ such that  $q_{x_1}^{\Lambda^{(1)} _{n}}(t)\rightrightarrows q_{x_1}(t)$ as $n\to\infty$, where $\rightrightarrows$ denotes the uniform convergence in $t\in[0,T]$.  Repeating these arguments, we can see that for any $k\in \mathbb{N}$ there exists a subsequence ${\bf\Lambda}^{(k)}=(\Lambda_{n}^{(k)})$ of ${\bf\Lambda}^{(k-1)}$ and $q_{x_k}(t)\in S$ such that  $q_{x_k}^{\Lambda _{n}^{(k)}}(t)\rightrightarrows q_{x_k}(t)$, $n\to\infty$. Then of course
$q_{x_m}^{\Lambda _{n}^{(k)}}(t)\rightrightarrows q_{x_m}(t)$, $n\to\infty$, for all $m\le k$.

Consider now the ``diagonal'' sequence ${\bf\Lambda}$ with elements $\Lambda_n:=\Lambda_{n}^{(n)}$. It is clear that $q_{x_k}^{\Lambda _{n}}(t)\rightrightarrows q_{x_k}(t)$ $n\to\infty$, for all $k\in \mathbb{N}$.

The limit transition in both sides of the equality
\begin{equation*}
q_{x}^{\Lambda _{n}}(t)=q_{x,0}^{\Lambda _{n}}+\int_{0}^{t}F_{x}\left(
q_{x}^{\Lambda _{n}}(s)\right) ds,\ x\in \gamma ,
\end{equation*}
shows that the functions $q_{x}(t)=\lim q_{x}^{\Lambda _{n}}(t),\ x\in \gamma $, solve
the system (\ref{ODE}). The inclusion
$\bar{q}(t)\in \bigcap_{\beta>j\alpha} S_{\beta }^\gamma$ and bounds (\ref{q-bound}) for all $\beta>j\alpha$
 follow from (\ref{q-est1111}). The proof is complete.
\end{proof}

\section{The uniqueness}
In this section, we will discuss conditions that guarantee the uniqueness of  the solution  ${\bar q}(t)\in S_{\beta }^{\gamma }$. We fix $\beta>0 $ and   let $\Delta _{\beta ,R}$ be the ball of  radius $
R>0$ in $S_{\beta }^{\gamma }$ centered at $0=(0)_{x\in\gamma}\in S_\beta^\gamma$.

For $x,y\in \gamma$ denote by $\frac{\partial F_{x}(\bar{q})}{\partial q_{y}}$  the Jacobian matrix of the mapping $F_x$ w.r.t. the variable $q_y$. By Condition~\ref{cond:row-finite}, this Jacobian is the zero matrix if $y\notin \gamma_{x,r}$.
We also define the corresponding gradient as the following vector with matrix components:
\begin{equation*}
\nabla F_{x}(\bar{q})=\left( \frac{\partial F_{x}(\bar{q})}{\partial q_{y}}
\right) _{y\in \gamma}, \quad x\in \gamma.
\end{equation*}
 Since all but finite number of the components of the gradient vector are zero matrices, we can define its norm
\[
	\bigl\bracevert\!\! \nabla F_{x}(\bar{q})\!\!\bigr\bracevert:=\sum_{y\in \gamma
}\left\Vert \frac{\partial F_{x}(\bar{q})}{\partial q_{y}}\right\Vert<\infty, \quad x\in \gamma,
\]
where $\lVert \cdot\rVert$ denotes the standard norm of a  linear operator in $S$.

To ensure the uniqueness, we assume the following:
\begin{condition}\label{cond:uniq}
For any $R>0$, there exists a constant $C_R>0$ such that
$$
\sup_{\bar{q}\in \Delta _{\beta ,R}}\bigl\bracevert\!\! \nabla F_{x}(\bar{q})\!\!\bigr\bracevert
  \leq
C_R\left( \lvert x\rvert +1\right), \quad x\in\gamma.  \label{eq:cond:uniq}
$$
\end{condition}

\begin{theorem}\label{theor-uniq}
Let $\beta>0$, and assume that Condition~\ref{cond:uniq} holds.
Suppose that the weight sequence $w(s)$ in  \eqref{defSalphagamma} satisfies the bound $w(s)\leq e^{\nu s}$, $s\in\R_+$ for some $\nu>0$.
Let
$\bar{q}^{(1)}(t),\bar{q}^{(2)}(t)\in S_\beta^\gamma$ be two pointwise solutions of \eqref{ODE} on $\left[0,T\right]$, and let $\bar{q}^{(1)}(0)=\bar{q}^{(2)}(0)$. Then  $\bar{q}^{(1)}(t)=\bar{q}^{(2)}(t)$, $t\in[0,T]$.
\end{theorem}

\begin{proof} We start by observing that, because of Condition~\ref{cond:row-finite}, we have  the bound
\begin{multline}\label{eq:look}
	\bigl\lvert F_x(\bar{q}^{(1)})-F_x(\bar{q}^{(2)})\bigr\rvert \\ \leq
	\sum_{y\in\gamma_{x,r}}\bigl\lvert q_{y}^{(1)}-q_{y}^{(2)}\bigr\rvert\sup_{\tau\in[0,1]}\biggl\lVert \frac{\partial F_{x}}{\partial q_{y}}\bigl(\tau\bar{q}^{(1)}+(1-\tau)\bar{q}^{(2)}\bigr) \biggr\rVert,
\end{multline}
which holds for any $\bar{q}^{(l)}=(q_x^{(l)})_{x\in \gamma}$, $l=1,2$.

Assume now that $\bar{q}^{(1)}, \bar{q}^{(2)}\in \Delta_{\beta,R}$, $R>0$. Then $\tau\bar{q}^{(1)}+(1-\tau)\bar{q}^{(2)}\in \Delta_{\beta,R}$, $\tau\in[0,1]$.
Let $n\geq 1$ and $x\in\gamma$ be such that $|x|\leq nr$. Then $|y|\leq (n+1)r$ for any $y\in\gamma_{x,r}$,  cf.~\eqref{eq:est1}. Thus  \eqref{eq:look} implies the following estimate:
\begin{equation}\label{finallyIgot}
\bigl\lvert F_x(\bar{q}^{(1)})-F_x(\bar{q}^{(2)})\bigr\rvert\leq \sup_{|y|\leq (n+1)r}\bigl\lvert q_y^{(1)}-q_y^{(2)}\bigr\rvert \  \sup_{\bar{q}\in \Delta _{\beta ,R}}\bigl\bracevert\!\! \nabla F_{x}(\bar{q})\!\!\bigr\bracevert.
\end{equation}
Let us now $\bar{q}^{(1)}(t),\bar{q}^{(1)}(t)\in S_\beta^\gamma$
be two  solutions of \eqref{ODE} on $t\in [0,T]$.
Set $R:=\max\limits_{t\in \left[ 0,T\right],\,l=1,2}\left\Vert \bar{q}^{(l)}(t)\right\Vert_\beta$, so that $\bar{q}^{(l)}(t)\in \Delta _{\beta,R}$, and let $C_R>0$ be such that Condition~\ref{cond:uniq} holds. Denote
\begin{equation}\label{eq:deltan}
\delta _{n}(t):=\sup_{\lvert x\rvert \leq
nr}  \bigl\lvert q_{x}^{(1)}(t)-q_{x}^{(2)}(t)\bigr\rvert, \quad n\geq 1, \
t\in[0,T].
\end{equation}
For any $x\in \gamma$ such that $\lvert x\rvert \leq nr$, it follows from the integral form of \eqref{ODE}, inequality \eqref{finallyIgot} and Condition~\ref{cond:uniq} that
\begin{align*}
\left\vert q_{x}^{(1)}(t)-q_{x}^{(2)}(t)\right\vert &\leq
\int_{0}^{t}\left\vert F_{x}(\bar{q}^{(1)}(s))-F_{x}(\bar{q}
^{(2)}(s))\right\vert ds \\
&\leq \int_{0}^{t}\delta _{n+1}(s) \sup_{\bar{q}\in \Delta _{\beta ,R}}\bigl\bracevert\!\! \nabla F_{x}(\bar{q})\!\!\bigr\bracevert ds \\
&\leq C_R\left( \lvert x\rvert +1\right) \int_{0}^{t}\delta
_{n+1}(s)ds.
\end{align*}
Therefore,
\begin{equation*}
\delta _{n}(t)\leq C_R\left( nr+1\right) \int_{0}^{t}\delta _{n+1}(s)ds\leq
\mu nr\int_{0}^{t}\delta _{n+1}(s)ds,
\end{equation*}
where $\mu:= C_R\bigl(1+\frac{1}{r}\bigr)>0$. The $N$-th iteration of this estimate gives
\begin{equation}\label{delta-recc}
\delta _{n}(t)\leq \frac{\left( t\mu r\right) ^{N}}{N!}n(n+1)...(n+N-1)\sup_{s\leq t}\delta _{n+N}(s).
\end{equation}
 A direct computation using \eqref{defSalphagamma}, \eqref{eq:deltan} and the inclusion $\bar{q}^{(1)}(t),\bar{q}^{(2)}(t)\in \Delta _{\beta ,R} $ shows that $\delta _{n+N}(s)\leq 2R w((n+N)r)^\beta$, $s\geq 0$. Using the assumption $w(s)\leq e^{\nu s}$, $s\in\R_+$, we get that
 \eqref{delta-recc} implies the bound
 \begin{align*}
	\delta _{n}(t) &\leq 2Re^{\beta\nu \left( n+N\right) r}(t\mu r)^N\binom{n+N-1}{N}\\
	&\leq 2Re^{\beta\nu n r}\biggl(t\mu re^{\beta\nu r+1}\frac{n+N-1}{N}\biggr)^N,
\end{align*}
where we used the well-known inequality $\binom{M}{N}\leq\bigl(\frac{M\,e}{N}\bigr)^N$, $1\leq N\leq M$. Therefore, for all $n\geq 1$ and $N> n-1$ we have
\begin{equation}\label{delta-ineq}
	\delta _{n}(t) < 2Re^{\beta\nu n r}\bigl(2t\mu re^{\beta\nu r+1}\bigr)^N, \quad t\geq0.
\end{equation}
Observe that for $t<t_0:=\bigl(2\mu re^{\beta\nu r+1}\bigr)^{-1}$ the r.h.s. of  inequality \eqref{delta-ineq} converges to zero as $N\rightarrow\infty$, which in turn implies that $\delta _{n}(t)=0$ for all $n\geq1$. Thus $\bar{q}^{(1)}(t)=\bar{q}^{(2)}(t)$ for $t\in [0,t_0)$.

These arguments can be repeated on each of the time intervals $[t_k,t_{k+1})$ with $t_k:=kt_0$, $k=1,2,...$, which shows that $\bar{q}^{(1)}(t)=\bar{q}^{(2)}(t)$, $t\in [0,T)$, for an arbitrary $T>0$.
\end{proof}

\section{Dynamics of interacting particle systems}\label{IAS}

Our main example is motivated by the study of (deterministic) dynamics of
interacting particle systems. In this case, $\gamma$ represents a collection of
particles indexed by elements $x$ of $X$, may be interpreted as particle positions. A particle with position $x\in \gamma$ carries an
internal parameter (spin) $\sigma _{x}\in S$. While the positions of our particles are fixed, the spins evolve according to system \eqref{ODE}. Here we will consider two types of the time-evolution with pair spin-spin interaction- general gradient-type dynamics and the Hamiltonian anharmonic dynamics. In the last subsection, we give  an example of somewhat different type, motivated by the study of self-organized systems.

\subsection{Gradient-type dynamics}\label{gtd}
Consider the following example of the family $(F_{x})_{x\in\gamma}$ that fulfills Condition~\ref{cond:row-finite}:
\begin{equation}
F_{x}(\bar{q}(t))=R_{x} (\bar{q}_{x,r}(t)) -\nabla U_{x}(q_{x}(t)), \quad
x\in\gamma,   \label{F}
\end{equation}
with
\begin{equation}
R_{x}(\bar{q})=\sum_{y\sim x, \, y\ne x}W_{xy}(q_{x},q_{y}),\quad x\in \gamma .
\label{R-pair}
\end{equation}
Here, for each $x\in\gamma$, we denote $\bar{q}_{x,r}(t):=\left( q_{y}(t)\right) _{y\in \gamma_{x,r}}$ and assume that $U_x\in C^2(S,\R_+)$, and hence $\nabla U_{x}\in C^{1}(S,S)$. Next, for each
$\{x,y\}\subset \gamma $, $x\sim y$, we assume that $W_{xy}\in C^{2}(S^{2},S)$.

\begin{proposition}\label{prop:suffforexist}
Let, for some $k\geq1$, $j\geq 2k$, $J_U>0$, condition \eqref{c01} holds, and condition \eqref{c0} is reinforced to
\begin{alignat}{2}
U_{x}(q)&\geq J_U\left\vert q\right\vert^{2k} , &&\qquad q\in S, \ x\in\gamma.\label{U}
\\\intertext{Let also, for some $J_W>0$,}
		\bigl\vert W_{xy}(q_1,q_2)\bigr\vert &\leq J_{W}\bigl( \left\vert
q_1\right\vert ^{k}+\left\vert q_2\right\vert ^{k}+1\bigr), &&\qquad q_1,q_2\in S, \ x,y\in\gamma.  \label{W}
\end{alignat}
Then there exists $C>0$ such that Condition~\ref{cond:diss} holds with $m=1$.
\end{proposition}
\begin{proof}
First we note that for $F$ of the form \eqref{F}  condition \eqref{eq:conddiss} reads as
\begin{equation}
R_{x}(\bar{q})\cdot \nabla U_{x}(q_{x})-\left\vert \nabla
U_{x}(q_{x})\right\vert ^{2}\leq C n_x^m \sum_{y\sim x} n_y^m U_y(q_y)
\label{c1-old}
\end{equation}
for any $\bar{q}\in S^{\gamma_{x,r}}$ and $x\in\gamma $. The inequality
$R_{x}\cdot \nabla U_{x}\leq \left\vert \nabla U_{x}\right\vert
^{2}+\left\vert R_{x}\right\vert ^{2}$ implies that the following bound is sufficient for \eqref{c1-old} to hold:
\begin{equation}
\left\vert R_{x}(\bar{q})\right\vert ^{2}\leq C  n_x^m  \sum_{y\sim x}n_y^m U_y(q_y), \quad \bar{q}\in S^\gamma, \ x\in\gamma.  \label{c11}
\end{equation}
We are going to check now \eqref{c11} with $m=1$, for $R_x$ given by \eqref{R-pair} under the conditions above.
 By \eqref{R-pair} and \eqref{W}, there exist $J_{1},J_{2},J_{3}>0$, such that
\begin{equation*}
\left\vert R_{x}(\bar{q})\right\vert ^{2}\leq J_{1}n_{x}\sum_{y\sim
x,\, y\neq x}\left\vert q_{y}\right\vert ^{2k}+J_{2}n_{x}^{2}\left\vert
q_{x}\right\vert ^{2k}+J_{3}n_{x}.
\end{equation*}
On the other hand, \eqref{U} yields
\begin{align*}
n_x \sum_{y\sim x } n_y U_{y}(q_{y})&\geq n_{x}\sum_{y\sim x,\, y\neq x}J_{U}\bigl(
\left\vert q_{y}\right\vert ^{2k}+1\bigr)
+n_{x}^{2}J_{U}\bigl( \left\vert q_{x}\right\vert ^{2k}+1\bigr) \\
&=J_{U}n_{x}\sum_{y\sim x, \, y\neq x}\left\vert q_{y}\right\vert
^{2k}+J_{U} n_{x}^{2} \left\vert q_{x}\right\vert ^{2k}+2J_{U}n_{x}^{2}-J_U n_x.
\end{align*}
Thus \eqref{c11} holds with $m=1$ and $C=\max \left\{
J_{1},J_{2},J_{3}\right\} /J_{U}$ (we used also that $n_x\geq1$ as $x\in\gamma_{x,r}$).
\end{proof}

\begin{proposition}\label{prop:uniq}
Assume that the interaction potentials $W$ and $U$ satisfy the bounds
\begin{alignat}{2}
\left\Vert \nabla ^{2}U_{y}(q)\right\Vert &\leq K_U \bigl(|q|^{j-2}+1\bigr), &&\quad q\in S, \ x\in \gamma, \label{Cond-un-self} \\
\left\Vert \frac{\partial }{\partial q_2}W_{xy}(q_1,q_1)\right\Vert&\leq K_W \bigl(|q_1|^{k-1}+|q_2|^{k-1}+1\bigr), &&\quad q_1,q_2\in S,\ x,y\in\gamma,  \label{Cond-un-pair}
\end{alignat}
for some $k\geq1$, $j\geq 2k$ and $K_U,K_W>0$. Let the pair $(w,z)$ satisfy Condition~\ref{cond:balance} and, moreover, the inequality
\begin{equation}
z(s)w(s)^{\beta(k-1)}+w(s)^{\beta(j-2)}\leq K_1 s+K_2,  \quad s\in\R_+.\label{eq:weights}
\end{equation}
holds for some $K_1,K_2>0$.
Then Condition~\ref{cond:uniq} holds.
\end{proposition}
\begin{proof}
Again,  we note that  for $F$ as in \eqref{F} Condition~\ref{cond:uniq} obtains the following form: for any $R>0$ there exists a constant $C_R>0$ such that
\begin{equation*}
\sup_{\bar{q}\in \Delta _{\beta ,R}}\sum_{y\in \gamma_{x,r} }\left\Vert
\frac{\partial R_{x}(\bar{q}_{x,r})}{\partial q_{y}}-  1\!\!1_{\{y=x\}} \nabla
^{2}U_{x}(q_{x})\right\Vert \leq C_R\left( \lvert x\rvert +1\right), \quad x\in\gamma,
\end{equation*}
where $\nabla^2$ is the Hessian matrix. Therefore, it is sufficient to assume that
\begin{equation}
\sup_{\bar{q}\in \Delta _{\beta ,R}}\Biggl( \sum_{y\in \gamma_{x,r}
}\left\Vert \frac{\partial R_{x}(\bar{q}_{x,r})}{\partial q_{y}}\right\Vert
+\left\Vert \nabla ^{2}U_{x}(q_{x})\right\Vert \Biggr) \leq C_R\left(
\lvert x\rvert +1\right), \quad x\in\gamma.\label{unique-pair}
\end{equation}

Let now $R_x$ be given by \eqref{R-pair}. Then, $\frac{\partial R_{x}(\bar{q}_{x,r})}{\partial q_{x}}$ is the zero matrix, and \eqref{unique-pair} can be rewritten, for each $x\in\gamma$, as follows
\begin{equation}
\sup_{\bar{q}\in \Delta _{\beta ,R}}\Biggl( \sum_{y\sim x,\, y\neq x
}\left\Vert \frac{\partial }{\partial q_{y}}W_{xy}(q_x,q_y)\right\Vert
+\left\Vert \nabla ^{2}U_{x}(q_{x})\right\Vert \Biggr) \leq C_R\left(
\lvert x\rvert +1\right).\label{unique-pair-indeed}
\end{equation}
Recall that $\bar{q}\in \Delta _{\beta ,R}$ implies
\[
	|q_x|\leq R w(|x|)^\beta, \quad x\in\gamma.
\]
Note that one can assume $R\geq1$. Then, by \eqref{Cond-un-pair}, \eqref{eq:condw}, \eqref{eq:nx}, we have
\begin{align*}
\sum_{y\sim x,\, y\neq x
}\left\Vert \frac{\partial }{\partial q_{y}}W_{xy}(q_x,q_y)\right\Vert&\leq K_W R^k
\sum_{y\sim x,\, y\neq x} \bigl(w(|x|)^{\beta (k-1)}+w(|y|)^{\beta (k-1)}+1\bigr)\\
&\leq  K_W R^k \bigl( n_x w(|x|)^{\beta (k-1)}+n_x w(|x|+r)^{\beta (k-1)}+1 \bigr)\\
&\leq K_W R^k(2+w_r^{\beta(k-1)}) z(|x|)w(|x|)^{\beta (k-1)}.
\end{align*}
Next, by \eqref{Cond-un-self},
\[
	\left\Vert \nabla ^{2}U_{x}(q_{x})\right\Vert\leq 2 K_U R^{j-2}w(|x|)^{\beta(j-2)}.
\]
As a result, by using \eqref{eq:weights}, we get \eqref{unique-pair-indeed}.
\end{proof}

\begin{remark}
A typical example of the pair interaction is
\[
	W_{xy}(q_x,q_y)=V(q_x-q_y), \qquad V\in C^2(S,S).
\]
Then, to ensure \eqref{W}, it is enough to assume that
\[
	|V(q)|\leq J_V(|q|^k+1), \quad q\in S
\]
for some $J_V>0$. Moreover, $\bigl\rVert\frac{\partial }{\partial q}V(q)\bigr\rVert=\bigl\rVert\frac{\partial }{\partial q_2}W(q_1,q_2)\bigr\rVert$ for $q=q_1-q_2$, and hence \eqref{Cond-un-pair} holds if only, for some $K_V>0$,
\[
	\biggl\rVert\frac{\partial }{\partial q}V(q)\biggr\rVert\leq K_V(|q|^{k-1}+1), \quad q\in S.
\]
\end{remark}

Now we will revisit our examples of admissible pairs and study the corresponding conditions of the existence and uniqueness of the corresponding dynamics.   Recall that the exponents $j$ and $k$ in \eqref{c01}, \eqref{W}, \eqref{Cond-un-self}, \eqref{Cond-un-pair} are always related by the assumption $j\geq 2k$, $k\geq1$.
\begin{example}[Example~\ref{ex:max} revisited]\label{ex:maxrev}
 Let \eqref{nz} hold. Then
Theorem~\ref{theor-main} holds with $w(s)=e^{s+e}$. Thus a solution with initial value
\[
|q_x(0)|\leq c \exp(\alpha|x|) , \quad x\in\gamma,
\]
will live in each of the spaces $S^\gamma_\beta$ and thus satisfy the bound
\[
|q_x(t)|\leq c(t) \exp(\beta|x|), \quad x\in\gamma,\,t>0,\label{eq:beta}
\]
for all $\beta>j\alpha$
(here $c(t)=c(\alpha,\beta,t)$).
It is clear that the uniqueness condition \eqref{eq:weights} of Proposition~\ref{prop:uniq} holds only for $k=1,\,j=2$.

\end{example}

\begin{example}[Example~\ref{ex:med} revisited]\label{ex:medrev}
Let \eqref{eq:nx} hold, $z$ satisfy \eqref{mednz}, and $w(s)=1+s$.
Assume without loss of generality that $j>2$.  Then for $1\leq k\leq \frac{j}{2}$ we have
\[
	z(s)w(s)^{\beta(k-1)}+w(s)^{\beta(j-2)}\leq z(s)(1+s)^{\frac{1}{2}}+(1+s),
\]
so that \eqref{eq:weights} holds in the corresponding $S^\gamma_\beta$ with any $\beta\le(j-2)^{-1}$. Thus we can apply both Theorems~\ref{theor-main} and~\ref{theor-uniq} with arbitrary $\alpha<j^{-1}(j-2)^{-1}$ and $\beta\in(j\alpha,(j-2)^{-1}]$. As a result,
for any initial data
\[
|q_x(0)|\leq c (1+|x|)^{\frac{1}{j(j-2)}}, \quad x\in\gamma,
\]
there exists the unique solution satisfying the estimate
\[
|q_x(t)|\leq c(t) (1+|x|)^{\frac{1}{j-2}}, \quad x\in\gamma.
\]
Moreover,  the stronger bound
\[
|q_x(t)|\leq c(t) (1+|x|)^{\frac{1}{\beta}}, \quad x\in\gamma.
\]
will hold for all $\beta\in(j\alpha,(j-2)^{-1}]$.
\end{example}

\begin{example}
A simple example of interactions satisfying assumptions \eqref{U}--\eqref{Cond-un-pair}
 is given by
\begin{equation*}
W_{xy}(q_{x},q_{y})=Jq_{y},\ J\in \R,
\end{equation*}
and
\begin{equation*}
U_{y}(q_{y})=aq_{y}^{2}-b,\ a,b>0.
\end{equation*}
In this case (\ref{ODE}) is just the linear system
\begin{equation*}
\dot{q}_{x}=-aq_{x}+J\sum_{y\sim x}q_{y}.
\end{equation*}
The existence and uniqueness of solutions in any $S_{\beta }^{\gamma }$,
which follows from Theorem~\ref{theor-main}, can be proved by a
direct application of Theorem~\ref{Th-ovs}.
\end{example}

\subsection{Infinite Anharmonic Systems}\label{ias}

Consider the Hamiltonian system
\begin{equation*}
\dot{q}_{x}=p_{x},\ \dot{p}_{x}=R_{x}(\bar{q})-\nabla U_{x}(q_{x}),\ x\in\gamma,
\end{equation*}
with $R_x$ and $U_x$ as in the previous section.
It fits the framework of Section~\ref{PDER} with the \textquotedblleft
double\textquotedblright\ state space $S\times S\ni (\bar{q},\bar{p})$ in
place of $S$ and
\begin{equation*}
F_{x}(\bar{q},\bar{p})=\left( \bar{p}_x,\left( R_{x}(\bar{q})-\nabla
U_{x}(q_{x})\right)\right) \in S\times S,\,  {x\in \gamma },
\end{equation*}
and the single-particle Hamiltonian
\begin{equation*}
H_{x}(q,p)=\frac{1}{2}\left\vert p\right\vert ^{2}+U_{x}(q)
\end{equation*}
replacing $U_x$ in Condition~\ref{cond:diss}.
\begin{proposition}
Assume that $R$ and $U$ satisfy \eqref{c11} resp. \eqref{unique-pair}. Then the modified version of Condition~\ref{cond:diss} resp. Condition~\ref{cond:uniq} holds.
\end{proposition}
\begin{proof}
We clearly have
\begin{equation*}
F_{x}(\bar{q},\bar{p})\cdot \nabla H_{x}(q_{x},p_{x})=p_{x}\cdot R_{x}(\bar{q}) \leq \frac{1}{2}\left( \left\vert p_{x}\right\vert ^{2}+\left\vert R_{x}(\bar{q})\right\vert ^{2}\right) .
\end{equation*}
Inequality (\ref{c11}) implies now that
\begin{multline*}
F_{x}(\bar{q},\bar{p})\cdot \nabla H_{x}(q_{x},p_{x})\leq \frac{1}{2}
\left\vert p_{x}\right\vert ^{2}+\frac{1}{2}Cn_{x}^{m}\sum_{y\sim
x}n_{y}^{m}U_{y}(q_{y}) \\
\leq \frac{1}{2}Cn_{x}^{m}\sum_{y\sim x}n_{y}^{m}H_{y}(q_{y},p_{y}),
\end{multline*}
and the modified Condition~\ref{cond:diss} holds. A direct
calculation shows that (\ref{unique-pair}) implies Condition~\ref{cond:uniq}.
\end{proof}

\subsection{Self-alignment dynamics}\label{sad}
This example is motivated by various models of self-organized dynamics based on the alignment. We refer to the review in \cite{MT} and references therein for the theory and applications of these models in biological, physical and social sciences. In this framework, equation  \eqref{ODE} takes the form
\begin{equation}\label{eq:self-align}
		\frac{d}{dt}q_x(t)=\sum_{y\sim x} g_{x,y}\bigl(\bar{q}_{x,r}(t)\bigr) \bigl(q_y(t)-q_x(t)\bigr),
\end{equation}
where $g_{x,y}\in C^1(S^{\gamma_{x,r}},\mathbb{R}),\,x,y\in\gamma$,  is a family of non-negative uniformly bounded  functions.

Note that in \cite{MT} the underlying set $\gamma$ is supposed to be finite and either fixed or allowed to evolve.
Thus \eqref{eq:self-align} can be considered as the infinite-particle (\textquotedblleft quenched\textquotedblright ) version of these models.

We consider two examples inspired by the so-called opinion dynamics, in which functions $g_{x,y}$ are given by either
\begin{equation}\label{eq:sym}
		g_{x,y}\bigl(\bar{q}_{x,r}\bigr)=\frac{\phi(|q_y-q_x|)}{n_x}, \qquad x\in\gamma,\ y\sim x,
\end{equation}
or
\begin{equation}\label{eq:nonsym}
		g_{x,y}\bigl(\bar{q}_{x,r}\bigr)=\dfrac{\phi(|q_y-q_x|)}{\sum\limits_{z\sim x}\phi(|q_z-q_x|)}, \qquad x\in\gamma,\ y\sim x,
\end{equation}
with $0<\phi(s)\leq\phi_0$, $s>0$ for some $\phi_0>0$.

Let us prove that Condition~\ref{cond:diss} holds with $U_x(q_x):=\frac{1}{2}|q_x|^2$, $x\in\gamma$. Indeed,
\begin{align*}
F_x(\bar{q})\cdot\nabla U_x(q_x)&=\sum_{y\sim x} g_{x,y}\bigl(\bar{q}_{x,r}\bigr) \bigl(q_y\cdot q_x-|q_x|^2\bigr)\\
&\leq \sum_{y\sim x} g_{x,y}\bigl(\bar{q}_{x,r}\bigr) \frac{1}{2}\bigl(|q_y|^2-|q_x|^2\bigr)\leq G \sum_{y\sim x} U_y(q_y),
\end{align*}
for some constant $G>0$, and hence \eqref{eq:conddiss} holds with $m=1$ and an arbitrary $n_x\geq1$, $x\in\gamma$. In particular, there are no further restrictions on the function $z$ in Condition~\ref{cond:balance}.

Next, to check Condition~\ref{cond:uniq}, we note that
\[
	\frac{\partial}{\partial q_y} F_x(\bar{q})=\sum_{y\sim x} g_{x,y}\bigl(\bar{q}_{x,r}\bigr) 1\!\!1 + \sum_{y\sim x} \frac{\partial}{\partial q_y}g_{x,y}\bigl(\bar{q}_{x,r}\bigr)\times q_y,
\]
where $u\times v$ denotes (for $u,v\in S=\mathbb{R}^\nu$) the matrix $(u_iv_j)_{1\leq i,j\leq \nu}$, and $1\!\!1$ is the identity matrix on $S$. We consider $g_{x,y}$ given by \eqref{eq:sym} and \eqref{eq:nonsym}, respectively.

In the case of \eqref{eq:sym} we have
\[
	\frac{\partial}{\partial q_y}g_{x,y}\bigl(\bar{q}_{x,r}\bigr)=\frac{\phi'(|q_y-q_x|)}{n_x}\frac{q_y}{|q_y|}.
\]
Observe that $\lVert q_y\times q_y\rVert\leq c |q_y|^2$, where the constant $c>0$ depends only on the choice of norm on $S=\mathbb{R}^\nu$. Thus we obtain the inequality
\[
	\biggl\lVert \frac{\partial}{\partial q_y} F_x(\bar{q})\biggr\rVert \leq \phi_0 + \sup_{s\geq 0}|\phi'(s)| \sum_{y\sim x} \frac{c |q_y|}{n_x}.
\]
Therefore, assuming that $|\phi'(s)|\leq \phi_1$ for $s\geq 0$, we see that
\[
	\bigl\bracevert\!\! \nabla F_{x}(\bar{q})\!\!\bigr\bracevert\leq \phi_0 n_x+
	\frac{c \phi_1 }{n_x}\|\bar{q}\|_\beta n_x w(|x|+r)^\beta,
\]
and hence Condition~\ref{cond:uniq} holds provided
\[
	z(s)+w(s)^\beta \leq K_1 s+K_2, \quad s\geq0,
\]
for some $K_1,K_2>0$.

In the case of $g_{x,y}$ given by \eqref{eq:nonsym} we obtain
\[
	\frac{\partial}{\partial q_y}g_{x,y}\bigl(\bar{q}_{x,r}\bigr)=\frac{\phi'(|q_y-q_x|)\sum\limits_{z\sim x,z\neq y}\phi(|q_z-q_x|)}{\Bigl(\sum\limits_{z\sim x}\phi(|q_z-q_x|)\Bigr)^2}\frac{q_y}{|q_y|}
\]
and proceed in a completely similar way provided  the condition of boundedness of $|\phi'(s)|$ is replaced by the bound $|\phi'(s)|\leq \phi_1 \phi(s)$, $s\geq0$.

Finally, one can consider the modification of \eqref{eq:self-align}, which corresponds to the so-called flocking dynamics \cite{MT}:
\begin{equation}\label{eq:flocking}
		\frac{d}{dt}q_x(t)=\sum_{y\sim x} g_{x,y}\bigl(\bar{p}_{x,r}(t)\bigr) \bigl(q_y(t)-q_x(t)\bigr), \qquad \frac{d}{dt}p_x(t)=q_x(t).
\end{equation}
In particular, for $g_{x,y}$ given by \eqref{eq:sym}, we will get an infinite-particle counterpart of the well-known Cucker--Smale dynamics, see e.g. \cite{CCR,MT}. It is straightforward to check that
\[
	U_x(q_x,p_x):=\frac{1}{2}|q_x|^2+\frac{1}{2}|p_x|^2
\]
fulfills Condition~\ref{cond:diss} and Condition~\ref{cond:uniq} holds for $g_{x,y}$ as in \eqref{eq:sym} or \eqref{eq:nonsym}.

\subsection*{Acknoledgements} We would like  to thank Sergio Albeverio, Yuri Kondratiev, Oles Kutovyi and Tanja Pasurek for their interest to this work and stimulationg discussions. Financial support of the Alexander von Humboldt Stiftung, the European Commission under the project STREVCOMS PIRSES-2013-612669 and DFG through SFB 1283 ``Taming uncertainty and profiting from randomness and low regularity in analysis, stochastics and their applications'' is gratefully appreciated.

\end{document}